\documentclass[a4paper,12pt]{amsart}
\usepackage[lmargin=3.7cm,rmargin=2.7cm,bmargin=3.5cm,tmargin=3cm]{geometry}

\usepackage[utf8]{inputenc}
\usepackage{amsmath}
\usepackage{amssymb}
\usepackage{amsthm}
\usepackage{booktabs}
\usepackage{textcomp}
\usepackage{enumerate}
\usepackage[sort,numbers]{natbib}
\usepackage{array}
\usepackage{psfrag}
\usepackage{graphicx}
\usepackage{bbm}
\usepackage{verbatim}
\usepackage[colorlinks]{hyperref}
\usepackage{color}

\newcommand{\defeq}{\mathrel{\mathop:}=}
\newcommand{\eqdef}{=\mathrel{\mathop:}}

\newcommand{\Lip}{\mathrm{Lip}}
\newcommand{\ev}{\operatorname{ev}}

\newcommand{\testconvex}{\mathcal{C}}
\newcommand{\testiconvex}{\mathcal{C}_+}
\newcommand{\bigvphantom}{\vphantom{\big|}}
\newcommand{\uarg}{\,\cdot\,}
\newcommand{\ud}{\mathrm{d}}

\newcommand{\R}{\mathbb{R}}
\newcommand{\Q}{\mathbb{Q}}
\newcommand{\N}{\mathbb{N}}

\newcommand{\pr}{\mathbb{P}}
\renewcommand{\P}{\mathbb{P}}
\newcommand{\E}{\mathbb{E}}

\newcommand{\mA}{\mathcal{A}}

\newcommand{\mE}{\mathcal{E}}
\newcommand{\mF}{\mathcal{F}}
\newcommand{\mG}{\mathcal{G}}

\newcommand{\mP}{\mathcal{P}}
\newcommand{\mM}{\mathcal{M}}
\newcommand{\mS}{\mathcal{S}}

\newcommand{\charfun}[1]{\mathbf{1}\left(#1\right)}
\newcommand{\ind}{\mathbf{1}}

\newcommand{\given}{\;:\;}
\newcommand{\bigmid}{\;\big|\;}

\newcommand{\lecx}{\le_{\mathrm{cx}}}
\newcommand{\lecxsp}{\le_{\mathrm{cx}\phantom{i}}}
\newcommand{\leicx}{\le_{\mathrm{icx}}}

\newcommand{\eqd}{\overset{\mathrm{d}}{=}} 
\newcommand{\deq}{\overset{\mathrm{d}}{=}} 
\newcommand{\graph}{\mathrm{graph}} 

\newtheorem{theorem}{Theorem}[section]

\newtheorem{proposition}[theorem]{Proposition}
\newtheorem{lemma}[theorem]{Lemma}

\newtheorem*{theorem*}{Theorem}
\newtheorem*{lemma*}{Lemma}
\newtheorem*{proposition*}{Proposition}
\newtheorem*{conjecture*}{Conjecture}

\theoremstyle{definition}

\theoremstyle{remark}
\newtheorem{remark}[theorem]{Remark}

\numberwithin{equation}{section}

\title{Conditional convex orders and measurable martingale couplings}

\author{Lasse Leskel\"a}
\address{Lasse Leskel\"a, 
  Department of Mathematics and Systems Analysis,
  School of Science, 
  Aalto University,
  PO Box 11100, 00076 Aalto, Finland}
\email{lasse.leskela@aalto.fi}

\author{Matti Vihola}
\address{Matti Vihola, Department of Mathematics and Statistics, University of
Jyv\"askyl\"a, P.O.Box 35, FI-40014 Univ.~of Jyv\"askyl\"a, Finland}
\email{matti.vihola@iki.fi}



\allowdisplaybreaks

\begin{document}

\begin{abstract} 
Strassen's classical martingale coupling theorem states that two
random vectors are ordered in the convex (resp.\ increasing convex)
stochastic order if and only if they admit a martingale (resp.\
submartingale) coupling. By analysing topological properties of spaces
of probability measures equipped with a Wasserstein metric and
applying a measurable selection theorem, we prove a conditional
version of this result for random vectors conditioned on a random
element taking values in a general measurable space. We provide an
analogue of the conditional martingale coupling theorem in the
language of probability kernels, and discuss how it can be applied in
the analysis of pseudo-marginal Markov chain Monte Carlo algorithms.
We also illustrate how our results imply the existence of a measurable
minimiser in the context of martingale optimal transport.
\end{abstract} 

\maketitle

\section{Introduction and main results}
\label{sec:intro} 

\subsection{Convex stochastic orders} 

Stochastic orders and relations provide powerful tools to compare
distributions of random variables and processes, and they have been
used in various applications
\cite{leskela-2010,shaked-shanthikumar,muller-stoyan,szekli}.  We
focus here on two closely related stochastic orders which are
characterised by expectations of convex functionals, the convex order
and the increasing convex order. The convex order is a common measure
of `variability' or `dispersion' of random variables and vectors, and
it arises naturally for example in majorisation \cite{marshall-olkin}.
The increasing convex order allows to compare also
random vectors with different means.

Let $\mu$ and $\nu$ be probability measures on $\R^d$. We say that $\mu$ is less than $\nu$ in the \emph{convex order}, denoted $\mu \lecx \nu$, if
\begin{equation}
 \label{eq:IntegralOrder}
 \int \phi \,\ud\mu \le \int \phi \,\ud\nu
\end{equation}
for all convex $\phi: \R^d \to \R_+$. We say that $\mu$ is less than
$\nu$ in the \emph{increasing convex order}, denoted $\mu \leicx \nu$,
if \eqref{eq:IntegralOrder} holds for all convex $\phi: \R^d
\to \R_+$ which are increasing with respect to the 
usual coordinate-wise partial order $x\le y$.

The following type of characterisation of convex orders in terms of
martingale couplings will be of our main interest. We denote by
$\mM_n(\R^d)$ (resp.\ $\mM^*_n(\R^d)$) the set of probability measures
$\lambda$ on $(\R^d)^n$ such that $\lambda$ is the joint distribution
of some $\R^d$-valued martingale (resp.\ submartingale) $(X_t)$ parametrised by $t
\in \{1,\dots,n\}$. Recall that a \emph{coupling} of probability
measures $\mu_1,\dots,\mu_n$ on $\R^d$ is a probability measure on
$(\R^d)^n$ having $\mu_1,\dots,\mu_n$ as its marginal distributions.

\begin{theorem}[Strassen \cite{strassen}] 
\label{th:strassen} 
For any probability measures $\mu$ and $\nu$ on $\R^d$ with finite first moments:
\begin{enumerate}[(i)]
\item $\mu \lecx \nu$ if and only if $\mu$ and $\nu$ admit a coupling
  $\lambda \in \mM_2(\R^d)$,
\item $\mu \leicx \nu$ if and only if $\mu$ and $\nu$ admit a coupling
  $\lambda \in \mM^*_2(\R^d)$.
\end{enumerate}
\end{theorem}

Stochastic orders are often expressed in the notation of random
variables instead of probability measures. Let $X$ and $Y$ be
random vectors on $\R^d$ defined on a probability space
$(\Omega,\mA,\pr)$. Then we denote $X \lecx Y$ (resp.\ $X \leicx Y$)
if the corresponding probability distributions $\pr \circ X^{-1}$ and
$\pr \circ Y^{-1}$ are ordered according to $\lecx$ (resp.\ $\leicx$),
that is,
\begin{equation}
 \label{eq:ExpectationOrder}
 \E \phi(X) \le \E \phi(Y)
\end{equation}
for all convex (resp.\ increasing convex) functions $\phi: \R^d \to
\R_+$.  Recall that a coupling of random vectors 
$X_1,\dots,X_n$ on $\R^d$ is a random vector 
$(\hat X_1,\dots, \hat X_d)$ defined on some probability space and
taking values in $(\R^d)^n$ such that $\hat X_i \deq X_i$ for all $i$, where
$\deq$ denotes equality in distribution. 
In this notation, Theorem~\ref{th:strassen} can be reformulated
as follows.

\begin{theorem}
\label{th:strassenRandom} 
For any real-valued random vectors $X$ and $Y$ with finite first moments:
\begin{enumerate}[(i)]
\item $X  \lecx Y$ if and only if $X$ and $Y$ admit a coupling $(\hat X, \hat Y)$ which satisfies $\hat X = \E( \hat Y \, | \, \hat X )$ almost surely.
\item $X \leicx Y$ if and only if $X$ and $Y$ admit a coupling $(\hat X, \hat Y)$ which satisfies $\hat X \le \E( \hat Y \, | \, \hat X )$ almost surely.
\end{enumerate}
\end{theorem}


\subsection{Main results} 

The main contribution of the present paper is the following theorem which extends the martingale characterisation
in Theorem~\ref{th:strassen} to pairs of probability measures indexed
by a parameter $\theta$ with values in some measurable space $S$.
Recall that a \emph{probability kernel} from $S$ to $\R^d$ is a map $P: (\theta,B) \mapsto P_\theta(B)$ such that
\begin{itemize}
\item $P_\theta$ is a probability measure on $\R^d$ for every
  $\theta\in S$, and
\item $\theta \mapsto P_\theta(B)$ is measurable for every Borel set $B \subset \R^d$.
\end{itemize}
We say that $P$ has \emph{finite first moments} if $\int |x| \,
P_\theta(\ud x) < \infty$ for all $\theta$.
We extend the notion of coupling to probability kernels as follows.  Let $P$ and $Q$
be probability kernels from $S$ to $\R^d$, and assume that $R$ is a
probability kernel from $S$ to $\R^d\times\R^d$. We say that $R$ is a \emph{pointwise
coupling} of $P$ and $Q$ if $R_\theta$ is a coupling of $P_\theta$ and
$Q_\theta$ for every $\theta$.

\begin{theorem} 
\label{th:strassenKernel} 
For any probability kernels $P$ and $Q$ from a measurable space $S$ to
$\R^d$ with finite first moments:
\begin{enumerate}[(i)]
\item $P_\theta \lecx Q_\theta$ for all $\theta$ if and only if $P$ and $Q$ 
  admit a pointwise coupling $R$ such that $R_\theta \in \mM_2(\R^d)$ for all $\theta$,
\item $P_\theta \leicx Q_\theta$ for all $\theta$ if and only if $P$ and $Q$ 
  admit a pointwise coupling $R$ such that $R_\theta \in \mM^*_2(\R^d)$ for all $\theta$.
\end{enumerate}
\end{theorem}

Conditional versions of integral stochastic orders may be defined by
considering conditional analogues of~\eqref{eq:ExpectationOrder}. Let
$Z$ be a random element with values in a measurable space $S$, defined
on the same probability space as random vectors $X$ and
$Y$ on $\R^d$. Then we denote $X \mid Z \lecx Y \mid Z$ (resp.\ $X \mid Z \leicx
Y \mid Z$) if
\[
 \E (\phi(X) \mid Z ) \le \E (\phi(Y) \mid Z) \quad \text{almost surely}
\]
for all convex (resp.\ increasing convex) functions $\phi: \R^d \to
\R$ such that $\phi(X)$ and $\phi(Y)$ are integrable. As a corollary
of Theorem~\ref{th:strassenKernel}, we will prove the following
conditional analogue of Theorem~\ref{th:strassenRandom}. Here a
\emph{$Z$-conditional coupling of $X$ and $Y$} is a random element
$(\hat X, \hat Y, \hat Z)$ such that $(\hat X, \hat Z) \deq (X,Z)$ and
$(\hat Y, \hat Z) \deq (Y,Z)$.

\begin{theorem}
\label{cor:cond-strassen-geni} 
For any real-valued random vectors $X$ and $Y$ with finite first moments and any random element $Z$
in a measurable space $S$:
\begin{enumerate}[(i)]
\item \label{item:cond-cx-strassen-geni}
  $X \mid Z \lecx  Y \mid Z$ if and only if $X$ and $Y$ admit a $Z$-conditional coupling $(\hat X, \hat Y, \hat Z)$ such
that $\hat X = \E( \hat Y \mid \hat X, \hat Z )$ almost surely.
\item \label{item:cond-icx-strassen-geni}
  $X \mid Z \leicx Y \mid Z$ if and only if $X$ and $Y$ admit a $Z$-conditional coupling $(\hat X, \hat Y, \hat Z)$ such
that $\hat X \le \E( \hat Y \mid \hat X, \hat Z )$ almost surely.
\end{enumerate}
\end{theorem}


\subsection{Related work} 

Theorem~\ref{th:strassen} extends by induction to the case where one
has countably many distributions $(\mu_n)_{n\in\N}$ with $\mu_n \lecx
\mu_{n+1}$ or $\mu_n \leicx \mu_{n+1}$. Kellerer \cite{kellerer}
extended this to the uncountable setting, by showing that a collection
of probability distributions parametrised by $t \in \R_+$ satisfies
$\mu_s \lecx \mu_t$ (resp.\ $\mu_s \leicx \mu_t$) for all $s \le t$ if
and only if there exists a martingale (resp.\ submartingale) $(X_t)$
with $X_t$ distributed according to $\mu_t$ for all $t \in \R_+$. This
relation is further explored in the recent monograph
\cite{hirsch-profeta-roynette-yor}; see also \cite{lowther}. The 'if'
part of Theorem~\ref{th:strassen} can be proved by a simple
application of Jensen's inequality, whereas the 'only if' part is more
subtle. Strassen's proof \cite[Theorems 8 and 9]{strassen} uses the
Hahn-Banach theorem. M\"uller and Stoyan \cite[Theorem 1.5.20 and
Corollary 1.5.21]{muller-stoyan} provide a more constructive proof,
still relying on a limiting argument. In fact, Strassen's
work \cite{strassen} addresses more general integral stochastic
orders, defined by requiring \eqref{eq:IntegralOrder} for a general
class of functions $\phi$. This allows to define orderings of random
variables with values in general measurable spaces, as further
investigated by Shortt \cite{shortt} and Hirshberg and Shortt
\cite{hirshberg-shortt}; see also Kertz and R\"osler
\cite{kertz-rosler}. Another direction of extending the theory of
stochastic orders is to consider nontransitive relations, see
Leskel\"a~\cite{leskela-2010}. Conditional stochastic orders have been
considered earlier more generally by R\"uschendorf \cite{ruschendorf},
following the work due to Whitt \cite{whitt,whitt-variability}.

The main result of this article (Theorem~\ref{th:strassenKernel})
extends Theorem~\ref{th:strassen} to parametrised collections of ordered pairs of probability distributions, in contrast with ordered sequences as in \cite{kellerer,hirsch-profeta-roynette-yor}. The proof of Theorem~\ref{th:strassenKernel} is based on measurability properties of related
set-valued mappings and an application of a measurable selection theorem of Kuratowski and Ryll-Nardzewski
\cite{kuratowski-ryll-nardzewski}. 
We are unaware of earlier results which would be directly applicable
in this context. However, similar results related to martingale
couplings have appeared recently in the context of optimal transport.
Beiglboeck and Juillet \cite{beiglboeck-juillet} consider the problem
of finding an optimal transport plan under the constraint that the
transport plan is a martingale. The work of Fontbona, Gu\'erin and
M\'el\'eard \cite{fontbona-guerin-meleard} has the most similarities with
our developments. With the notation above, they consider finding a
measurable optimal transport plan between $P_\theta$ and
$Q_\theta$. The work of Hobson \cite{hobson}, brought to our attention by a
referee, provides an explicit Skorokhod embedding of two univariate
convex ordered distributions. This embedding could be used to prove our
result in the scalar case.


\subsection{Outline of the rest of the paper} 

Section~\ref{sec:cond-conv} discusses the definitions and basic
properties related to conditional convex stochastic orders. The proofs
of Theorems~\ref{th:strassenKernel} and \ref{cor:cond-strassen-geni}
are given in Section~\ref{sec:proof} after analysing the measurability
of related set-valued mappings.

Our problem was initially motivated by applied work on so-called 
pseudo-marginal Markov chain Monte Carlo algorithms
\cite{andrieu-roberts}. In Section \ref{sec:application}, we 
summarise the application and discuss why a martingale coupling
is crucial in this context. We discuss in Section \ref{sec:extensions} some
extensions of our results and their applicability 
in the context of martingale optimal transport.



\section{Conditional convex orders} 
\label{sec:cond-conv} 

\subsection{Definitions and basic properties} 

We denote the $d$-dimensional Euclidean space by $\R^d$, the real line
by $\R^1=\R$ and the set of positive real numbers by $\R_+$. We follow
the convention that a number $x$ is \emph{positive} if $x \ge 0$ and a
function $f$ is \emph{increasing} if $f(x) \le f(y)$ for all $x \le
y$, with the usual coordinate-wise partial order, which holds if all
the coordinates are ordered by $x_i\le y_i$ for $1\le i\le d$. Unless
otherwise mentioned, all measures on a topological space will
considered as measures defined on the corresponding Borel
sigma-algebra. A random vector $X$ is called \emph{integrable} if $\E
|X| < \infty$. When $X$ and $Y$ are integrable, it is not hard to
verify that $X \lecx Y$ (resp.\ $X \leicx Y$) if and only if
\eqref{eq:ExpectationOrder} holds for all convex (resp.\ increasing
convex) $\phi: \R^d \to \R$ such that $\phi(X)$ and $\phi(Y)$ are
integrable.

The following definition extends the $Z$-conditional order in Section \ref{sec:intro} to an order conditioned on a sigma-algebra. 
Let $X$ and $Y$ be integrable random variables defined on a probability space
$(\Omega,\mA, \P)$, and let $\mF\subset\mA$ be a sigma-algebra.  We denote
$X \mid \mF \lecx Y \mid \mF$ (resp.\ $X \mid \mF \leicx Y \mid \mF$) if
\[
 \E (\phi(X) \mid \mF ) \le \E (\phi(Y) \mid \mF) \quad \text{almost surely}
\]
for all convex (resp.\ increasing convex) functions $\phi: \R^d \to \R$ such that $\phi(X)$ and $\phi(Y)$ are integrable.
When this is the case we say that $X$ is less than $Y$ in the
\emph{conditional convex} (resp.\ \emph{increasing convex})
\emph{order given $\mF$}. In the special case when $\mF = \sigma(Z)$
is generated by a random element $Z$ with values in some measurable
space, we write $X \mid Z \lecx Y \mid Z$ and $X \mid Z \leicx Y \mid Z$.

We state next a proposition which suggests that conditional
convex orders can be seen as interpolations between (unconditional) convex
orders and the corresponding strong stochastic orders.
\begin{proposition} 
Let $X$ and $Y$ be integrable random vectors defined on $(\Omega,\mA,\P)$ and let
$\mF \subset \mG$ be subsigma-algebras of $\mA$.
\begin{enumerate}[(i)]
\item \label{item:nested-icx} $X \mid \mG \leicx Y \mid \mG \implies
  X \mid \mF \leicx Y \mid \mF \implies X \leicx Y$.
\item \label{item:nested-cx} $X\mid \mG \lecxsp Y \mid \mG \implies
  X \mid \mF \lecxsp Y \mid \mF \implies X\lecxsp Y$.
\item \label{item:trivial-icx} $X\mid \mA \leicx Y \mid \mA \iff  X \le Y$ almost surely.
\item \label{item:trivial-cx} $X\mid \mA \lecxsp Y \mid \mA \iff X = Y$ almost surely.
\end{enumerate}
\end{proposition} 
\begin{proof} 
For \eqref{item:nested-icx} assume that $X \mid \mG \leicx Y \mid \mG$, and let $\phi$ be an increasing convex function such that $\phi(X)$ and $\phi(Y)$ are integrable. Then by the tower property of conditional expectations,
\[
 \E (\phi(X) \mid \mF)
 =  \E [ \E (\phi(X) \mid \mG) \mid \mF ]
 \le \E [ \E (\phi(Y) \mid \mG) \mid \mF ]
 = \E (\phi(Y) \mid \mF)
\]
almost surely. Therefore $X \mid \mF \leicx Y \mid \mF$. The second implication in \eqref{item:nested-icx} follows by writing the above inequality for $\mF = \{\emptyset,\Omega\}$.
Part 
\eqref{item:nested-cx} follows similarly.

Part \eqref{item:trivial-icx} is direct, and
for \eqref{item:trivial-cx}, notice that $X \mid \mA \lecx Y \mid \mA$
implies $X \mid \mA \leicx Y \mid \mA$ and $-X \mid \mA \leicx -Y \mid \mA$. 
By \eqref{item:trivial-icx} we conclude that $X=Y$ almost surely. 
The reverse implication is trivial.
\end{proof} 

\subsection{Countable characterisations} 

Instead of testing the expectations of all (increasing) convex functions, the
following lemma states that it is enough to restrict to a countable
family of such functions.
\begin{lemma} 
    \label{lem:countable-char} 
There exist countable sets of convex functions $\testconvex$ and 
increasing convex functions $\testiconvex$ such that
\begin{align*}
    X &\lecx Y &\iff &  &\E \phi(X) &\le \E \phi(Y) &&
    \text{for all $\phi\in\mathcal{C}$}, \\
    X &\leicx Y &\iff &  &\E \phi(X) &\le \E \phi(Y) &&
    \text{for all $\phi\in\mathcal{C}_+$}.
\end{align*}
\end{lemma} 
The proof of Lemma \ref{lem:countable-char} is given in Appendix
\ref{sec:countable-char}.

In the univariate case, Lemma \ref{lem:countable-char} follows 
from the following well-known characterisations 
\cite[Theorems 3.A.2 and 4.A.2]{shaked-shanthikumar}.
Here $(x)_+\defeq \max\{0,x\}$ denotes the positive part of a number $x$.
\begin{proposition} 
    \label{prop:order-char} 
    Let $X$ and $Y$ be integrable random variables. Then
\begin{align*}
   X&\lecxsp Y &\iff& &\E|X-t| &\le \E|Y-t| &&\text{for all $t\in\R$,}
   \\
   X&\leicx Y &\iff& &\E(X-t)_+&\le \E(Y-t)_+
   &&\text{for all $t\in\R$.}
\end{align*}
\end{proposition} 

\begin{remark}
    \label{rem:order-char} 
    It is easy to see that we may
restrict to $t\in \Q$ in Proposition \ref{prop:order-char}, implying
that in the univariate case, we may take $\testconvex = \{x\mapsto
|x-t|\given t\in\Q\}$ and $\testiconvex = \{x\mapsto (x-t)_+ \given
t\in \Q\}$ in Lemma \ref{lem:countable-char}. 
\end{remark} 

The characterisations in Proposition \ref{prop:order-char} are often
easier to check in practice. In the insurance context, the quantity
$\E(X-t)_+$ has an interpretation as a stop-loss
\cite{borglin-keiding}. Unfortunately, such simple parametrisations
are not available in the multivariate case; see the discussion in
\cite[p.~98]{muller-stoyan}.
Both Lemma~\ref{lem:countable-char} and Proposition
extend naturally to the conditional case; see
Lemma~\ref{lem:cond-countable-char} and 
Proposition~\ref{prop:cond-simple-char}.


\subsection{Characterisations using regular conditional distributions} 

If $X$ is a real-valued random vector defined on a probability space
$(\Omega,\mA,\pr)$ and $\mF \subset \mA$ is a sigma-algebra, recall
that a \emph{regular conditional distribution of $X$ given $\mF$} is a
map $(\omega,B) \mapsto P^\mF_\omega(B)$ such that $P^\mF_\omega$ is a
probability measure on $\R^d$ for every $\omega$, and $\omega \mapsto
P^\mF_\omega(B)$ is a version of $\E ( \ind(X \in B) \mid \mF )$ for
every Borel set $B \subset \R^d$.  Hence $P^\mF$ is a random probability
measure, and the probability that $P^\mF$ assigns to a Borel set $B$
is an $\mF$-measurable random variable with expectation $\pr(X \in
B)$. If $P^\mF$ is a regular conditional distribution of a $X$ given
$\mF$, then
\begin{equation}
 \label{eq:disintegration}
 \E (\phi(X) \mid \mF ) = \int \phi(x) \, P^\mF(\ud x)
\end{equation}
almost surely for any $\phi$ such that $\phi(X)$ is integrable \cite[Thm~6.4]{kallenberg}.

The next result shows that conditional convex orders can be expressed
equivalently by the corresponding orders of the related conditional distributions.

\begin{proposition} 
    \label{prop:cond-cx-char} 
Assume that $X$ and $Y$ are integrable random vectors
defined on a probability space $(\Omega,\mA,\P)$, and let $\mF\subset\mA$ be a
sigma-algebra. Let $P^\mF$ and $Q^\mF$ stand for
regular conditional distributions of $X$ and $Y$ given $\mF$,
respectively.
Then, 
\begin{align}
X\mid \mF&\lecxsp Y\mid \mF & \iff&& P^\mF &\lecxsp Q^\mF
 && \text{almost surely,}\tag{i} 
 \label{eq:cond-cx-char} \\
X \mid \mF&\leicx Y\mid \mF &\iff&& P^\mF &\leicx Q^\mF 
&& \text{almost surely.}\tag{ii}
  \label{eq:cond-icx-char}
\end{align}
\end{proposition} 
\begin{proof} 
Assume first that $P^\mF \lecx Q^\mF$ almost surely. Let $\phi: \R^d \to \R$ be a convex
function such that $\phi(X)$ and $\phi(Y)$ are integrable. Then by~\eqref{eq:disintegration},
\[
 \E(\phi(Y)\mid \mF ) - \E(\phi(X)\mid \mF)
 \ = \ \int \phi(y) Q^\mF(\ud y) - \int \phi(x) P^\mF(\ud x)
 \ \ge \ 0
\]
almost surely. As a consequence, $X \mid \mF \lecx Y \mid \mF$.

To prove the converse in \eqref{eq:cond-cx-char}, assume that $X \mid \mF \lecx Y \mid \mF$.
Let $\Omega_0$ be the event that $P^\mF$ and $Q^\mF$ have
finite first moments. Then $\pr(\Omega_0)=1$. 
Recall Lemma \ref{lem:countable-char}, fix a function $f\in\testconvex$ 
and define
\[
 Z_f(\omega) = \int f(x) \, Q^\mF_\omega(\ud x) - \int f(x) \,
 P^\mF_\omega(\ud x)
\]
for $\omega \in \Omega_0$, and let $Z_f(\omega) = 0$ otherwise. 
Then by~\eqref{eq:disintegration}, 
\[
    Z_f = \E( f(Y) \mid \mF ) - \E( f(X) \mid \mF )\ge 0
\]
almost surely. This further implies that $\inf_{f \in \testconvex} Z_f \ge 0$ 
almost surely. We conclude from Lemma \ref{lem:countable-char} that 
$P^\mF \lecx Q^\mF$ almost surely.

The proof if \eqref{eq:cond-icx-char} is identical, except 
with functions $f\in\testiconvex$.
\end{proof} 

Let us now consider the case where the sigma-algebra $\mF = \sigma(Z)$
is generated by a random element $Z$ taking values in a general
measurable space $S$. Then for any random vector $X$
defined on the same probability space as $Z$ there exists
\cite[Thm~6.3]{kallenberg} a probability kernel $P$ from $S$ to $\R$
such that $\omega \mapsto P_{Z(\omega)}(B)$ is a version of $\E(\ind(X
\in B) \, | \, Z)$ for every Borel set $B \subset \R^d$. Such $P$ is
called a \emph{regular conditional distribution of $X$ given $Z$}, and
we note that $(\omega,B) \mapsto P_{Z(\omega)}(B)$ is a regular
conditional distribution of $X$ given $\sigma(Z)$ in the sense defined
in the beginning of the section. In this case the conditional convex
and increasing convex orders can be characterised as follows.

\begin{proposition} 
    \label{cor:cond-cx-char} 
Let $X$ and $Y$ be integrable random vectors and $Z$ a random
element in a measurable space $S$, all defined on a common probability space. If $P$ and $Q$ are regular conditional distributions of $X$ and $Y$ given $Z$, then
\begin{align}
 X \mid Z&\lecxsp Y\mid Z & \iff&& 
 P_\theta &\lecxsp Q_\theta && 
   \text{for $\mu$-almost every $\theta \in S$,} \tag{i} 
   \label{eq:cor-cond-cx-char} \\
 X \mid Z &\leicx Y\mid Z &\iff&& 
 P_\theta &\leicx Q_\theta && 
 \text{for $\mu$-almost every $\theta \in S$,} \tag{ii}
 \label{eq:cor-cond-icx-char}
\end{align}
where $\mu$ stands for the distribution of $Z$.
\end{proposition} 
\begin{proof} 
Let $\mF = \sigma(Z)$ and denote $P^\mF_\omega(B) = P_{Z(\omega)}(B)$
and $Q^\mF_\omega(B) = Q_{Z(\omega)}(B)$ for $\omega \in \Omega$ and
Borel sets $B \subset \R$. Then $P^\mF$ and $Q^\mF$ are regular
conditional distributions of $X$ and $Y$ given $\mF$, respectively.
Let $ S_0 = \{ \theta \in S: P_\theta \lecx Q_\theta\}. $ The argument
used in the proof of Proposition~\ref{prop:cond-cx-char} shows that
\begin{equation*}
 S_0 = \bigcap_{f \in \testconvex} \left\{ \theta \in S: \int f(x) \,
   P_\theta(\ud x) \le \int f(y) \, Q_\theta(\ud y) \right\},
\end{equation*}
from which we conclude that $S_0$ is a measurable subset of $S$.
Proposition~\ref{prop:cond-cx-char} now tells us that $X \mid Z
\lecxsp Y\mid Z$ if and only if $P_{Z(\omega)} \lecx Q_{Z(\omega)}$
for $\pr$-almost every $\omega$. The latter condition is equivalent to
requiring that $\mu(S_0) = \pr(Z \in S_0) = 1$. Hence we have proved
claim \eqref{eq:cor-cond-cx-char}. The proof of claim
\eqref{eq:cor-cond-icx-char} is analogous.
\end{proof}

As another corollary of Proposition \ref{prop:cond-cx-char} we obtain
the following conditional version of
Lemma~\ref{lem:countable-char}.

\begin{lemma} 
    \label{lem:cond-countable-char} 
 Let $X$ and $Y$ be integrable random vectors defined on a
 probability space $(\Omega,\mA,\P)$, and let $\mF \subset \mA$ be a 
 sigma-algebra. Then, there exist countable sets of convex functions 
 $\testconvex$ and $\testiconvex$ such that
\begin{align}
   X\mid \mF &\lecxsp Y\mid \mF &\iff& &
   \E\big[f(X)\bigmid \mF \big] &\le \E\big[f(Y)\bigmid \mF \big],
   &&f\in\testconvex, 
   \tag{i}
   \label{eq:cond-cx-countable} \\
   X \mid \mF &\leicx Y \mid \mF &\iff& &
   \E\big[f(X)\bigmid  \mF \big] 
   &\le \E\big[f(Y)\bigmid \mF \big],
   &&f\in\testiconvex,
   \tag{ii}
   \label{eq:cond-icx-countable}
\end{align}
where the inequalities on the right hold almost surely
for any $f\in\testconvex$ or $f\in\testiconvex$.
\end{lemma} 
\begin{proof} 
The forward directions of both claims follow trivially, as 
$f\in\testconvex$ are convex and $f\in\testiconvex$ are increasing
convex functions.

For the opposite direction, assume that the inequality on the right of
\eqref{eq:cond-cx-countable}
holds for all $f\in\testconvex$ almost surely. Let $P^\mF$ and $Q^\mF$ 
be regular
conditional distributions of $X$ and $Y$ given $\mF$, respectively. Then
\begin{equation}
 \label{eq:cxCondReg}
 \int f(x)  P^\mF(\ud x) \le \int f(y) Q^\mF(\ud y)
\end{equation}
almost surely for all $f\in\testconvex$. Let $\Omega_0$ be 
the event that \eqref{eq:cxCondReg}
holds for all $f\in\testconvex$, then $\pr(\Omega_0) = 1$. 
Lemma \ref{lem:countable-char} hence implies that 
$P^\mF_\omega \lecx Q^\mF_\omega$ for all $\omega \in \Omega_0$, and Proposition \ref{prop:cond-cx-char} shows that $X \mid \mF \lecx Y\mid \mF$.
The opposite direction of claim \eqref{eq:cond-icx-countable} 
is proved in a similar way.
\end{proof} 

We also state the conditional version of 
Proposition~\ref{prop:order-char}, which follows from Lemma
\ref{lem:cond-countable-char} as suggested in Remark \ref{rem:order-char}.
\begin{proposition} 
    \label{prop:cond-simple-char} 
 Let $X$ and $Y$ be integrable random variables defined on a
 probability space $(\Omega,\mA,\P)$, and let $\mF \subset \mA$ be a 
 sigma-algebra. Then,
\begin{align*}
   X\mid \mF &\lecxsp Y\mid \mF &\iff& &
   \E\big[|X-t|\;\big|\; \mF \big] &\le \E\big[|Y-t|\;\big|\; \mF \big],
   \\
   X \mid \mF &\leicx Y \mid \mF &\iff& &
   \E\big[(X-t)_+\;\big|\; \mF \big] &\le \E\big[(Y-t)_+\;\big|\; \mF \big],
\end{align*}
where the inequalities on the right hold almost surely
for any $t\in\R$. 
\end{proposition} 



\section{Proofs of the main results} 
\label{sec:proof} 

This section is devoted to proving Theorems~\ref{th:strassenKernel}
and~\ref{cor:cond-strassen-geni}. Our proof of
Theorem~\ref{th:strassenKernel} is based on a measurable selection
theorem of Kuratowski and Ryll-Nardzewski 
\cite{kuratowski-ryll-nardzewski}. To apply it, we first need to
analyse the regularity of coupling constructions and probability
kernels with respect to suitable measurable structures on spaces of
probability measures. Because convex orders are essentially restricted
to probability measures with finite first moments, our natural choice
is to consider Borel sigma-algebras generated by the Wasserstein
metric which will be discussed in Section~\ref{sec:Wasserstein}. A
similar measurability analysis for the topology corresponding to
convergence in distribution has been carried out in
\cite{leskela-2010}. The space of martingale distributions with
respect to the Wasserstein metric is analysed in
Section~\ref{sec:martingales}, whereas Section~\ref{sec:set-valued}
establishes crucial measurability properties of probability kernels
and marginalising maps.  Section~\ref{sec:proof-conclusion} concludes
the proof of Theorem \ref{th:strassenKernel} and
Section~\ref{sec:proof-other} concludes the proof of
Theorem~\ref{cor:cond-strassen-geni}.

\subsection{Wasserstein metric}
\label{sec:Wasserstein} 

For a probability measure $\mu$ on $S$
and a measurable function $f:S \to S'$,
we denote by $f_\# \mu = \mu \circ f^{-1}$ the pushforward measure of 
$\mu$ by $f$. When $S = S_1 \times \cdots \times S_d$, we denote the $i$-th coordinate projection
by $\pi^i(x_1,\dots,x_d) \defeq x_i$. Then $\pi^i_\# \mu$ equals the
$i$-th marginal distribution of $\mu$.
The set of couplings of $\mu \in \mP(S_1)$ and $\nu \in \mP(S_2)$ will be denoted by
\begin{equation*}
  \Gamma(\mu,\nu) \defeq \{\lambda\in
  \mP(S_1 \times S_2)\given 
  \pi_\#^1 \lambda = \mu,\, 
  \pi_\#^2 \lambda = \nu\}.
\end{equation*}

Let us recall the definition of the Wasserstein
(a.k.a.~Kantorovich-Rubinstein) metric between two
probability measures $\mu, \nu \in \mP_1(\R^d)$:
\[
 d_W(\mu,\nu) \defeq \min_{\lambda \in \Gamma(\mu,\nu)} \int_{\R^d
   \times \R^d} |x-y| \, \lambda(\ud x\times\ud y).
\]
The minimum is attained by lower semicontinuity properties and the
relative compactness of $\Gamma(\mu,\nu)$, and the map $d_W$ is a
metric on $\mP_1(\R^d)$ \cite[Section 7.1]{ambrosio-gigli-savare}.

The space $\mP_1(\R^d)$ equipped with the Wasserstein metric is a
complete separable metric space \cite[Proposition
7.1.5]{ambrosio-gigli-savare}.
The same proposition also shows
that $d_W(\mu_n,\mu) \to 0$ if and only if $\mu_n \to \mu$ 
in distribution and
$(\mu_n)$ is uniformly integrable in the sense that
\[
 \sup_{n} \int_{\R^d} |x| \charfun{|x| > t} \mu_n(\ud x) \to 0
 \qquad \text{as $t \to \infty$}.
\]
Hereafter, we equip 
$\mP_1(\R^d)$ by the topology induced by $d_W$.

The following results are probably well-known in transport theory,
but we were unable to find them in the literature.
We provide proofs for the reader's convenience.

\begin{lemma} 
\label{the:ContinuousProjection} 
The $i$-th marginal map $\pi_\#^i: \mP_1\big((\R^d)^n\big) \to \mP_1(\R^d)$ is
continuous for all~$i$.
\end{lemma} 
\begin{proof} 
Assume that $\mu_n \to \mu \in \mP_1\big((\R^d)^n\big)$. Then 
$\mu_n \to \mu$ in
distribution and $(\mu_n)$ is uniformly integrable. If $f:\R^d \to \R$
is continuous and bounded, then so is $f \circ \pi_i: (\R^d)^n \to \R$.
Therefore, $(\pi_\#^i \mu_n)(f) = \mu_n( f \circ \pi_i) \to \mu(f
\circ \pi_i) = (\pi_\#^i \mu)(f)$. Thus, $\pi_\#^i \mu^n \to \pi_\#^i
\mu$ in distribution. It is also easy to see that $(\pi_\#^i \mu_n)$
is uniformly integrable because
\begin{align*}
 \int_{\R^d} |x_i| \charfun{|x_i|>t}  \pi_\#^i \mu_n(\ud x_i)
 &= \int_{(\R^d)^n} |x_i| \charfun{|x_i|>t}  \mu_n(\ud x) \\
 &\le \int_{(\R^d)^n} |x| \charfun{|x|>t}  \mu_n(\ud x).\qedhere
\end{align*}
\end{proof} 

\begin{lemma}
\label{the:CouplingCompact} 
For any $\mu,\nu \in \mP_1(\R^d)$, the set of couplings 
$\Gamma(\mu,\nu)$ is compact in $\mP_1(\R^d\times\R^d)$.
\end{lemma} 
\begin{proof} 
Let $\lambda \in \Gamma(\mu,\nu)$. Note that $|(x,y)|/2\le
\max\{|x|,|y|\} \eqdef
|x|\vee |y|$ for all $x,y\in \R^d$. Therefore, for any $t>0$
\begin{align*}
    \frac{1}{2}\int |(x,y)| &\charfun{|(x,y)|>t} \lambda(\ud x\times
    \ud y) \\
    &\le\int \big(|x| \vee |y|\big) \charfun{2\big(|x|\vee |y|\big)>t} \lambda(\ud x\times
    \ud y) \\
    &\le \int |x| \charfun{|x|>\smash{\frac{t}{2}}\vphantom{\Big|}} \mu(\ud x)
    + \int |y| \charfun{|y|>\smash{\frac{t}{2}}\vphantom{\Big|}} \nu(\ud y) .
\end{align*}
Because the measures $\mu$ and $\nu$ have finite first moments,
the right side above tends to zero as $t \to \infty$, uniformly with 
respect to $\lambda \in \Gamma(\mu,\nu)$
We conclude that $\Gamma(\mu,\nu)$ is uniformly integrable and hence also tight. By
\cite[Proposition 7.1.5]{ambrosio-gigli-savare}, it follows that
$\Gamma(\mu, \nu)$  is relatively compact in $\mP_1(\R^d\times\R^d)$.

To verify that $\Gamma(\mu,\nu)$ is closed, it suffices to observe that
it can be written as a preimage $ \Gamma(\mu, \nu) =  
\Pi^{-1}\big(\{(\mu,\nu)\}\big)$
of the map $\Pi: \mP_1(\R^d\times\R^d) \to \mP_1(\R^d)\times\mP_1(\R^d)$ defined by $\Pi(\lambda) = (\pi^1_\# \lambda, \pi^2_\# \lambda)$
which is continuous by Lemma~\ref{the:ContinuousProjection}.
\end{proof}


\subsection{Two-parameter martingales and submartingales}
\label{sec:martingales} 

Recall that $\mM_2(\R^d)$ (resp.\ $\mM^*_2(\R^d)$) denotes the
collection of probability measures on $\R^d\times\R^d$ which are joint
distributions of a two-parameter martingale (resp.\ submartingale).
The following elementary lemmas stated without a proof 
give convenient ways to characterise
these collections. 


\begin{lemma} 
\label{lem:mart-char} 
The following are equivalent for any $\lambda \in \mP_1(\R^d\times\R^d)$:
\begin{enumerate}[(i)]
\item \label{item:cond-Mstari}
$\lambda \in \mM_2(\R^d)$.
\item \label{item:cond-exp-marti}
$\E[Y \mid X] = X$ a.s.\ for any random vector $(X,Y)$ with distribution 
$\lambda$.
\item \label{item:set-test-marti}
$\int y \ind(x\in A) \lambda(\ud x \times \ud y) 
= \int x \ind(x\in A) \lambda(\ud x \times \ud y)$ for all Borel sets 
$A \subset \R^d$. 
\item \label{item:cont-test-marti}
$\int y \phi(x) \lambda(\ud x\times\ud y) \! = \! \int x \phi(x)
\lambda(\ud x\times\ud y)$ for all continuous bounded $\phi: \R^d \to \R_+$.
\end{enumerate}
\end{lemma} 

\begin{lemma} 
\label{lem:submart-char} 
The following are equivalent for any $\lambda \in \mP_1(\R^d\times\R^d)$:
\begin{enumerate}[(i)]
\item \label{item:cond-Mstar}
$\lambda \in \mM^*_2(\R^d)$.
\item \label{item:cond-exp-mart}
$\E[Y \mid X] \ge X$ a.s.\ for any random vector $(X,Y)$ with distribution $\lambda$.
\item \label{item:set-test-mart}
$\int y \ind(x\in A) \lambda(\ud x \times \ud y) \ge \int x \ind(x\in A)
\lambda(\ud x \times \ud y)$ for all Borel sets $A \subset \R^d$. 
\item \label{item:cont-test-mart}
$\int y \phi(x) \lambda(\ud x\times\ud y) \!\ge\! \int x \phi(x)
\lambda(\ud x\times\ud y)$ for all continuous bounded $\phi:\R^d\to\R_+$.
\end{enumerate}
\end{lemma} 

The following lemma shows that martingale and submartingale measures
form closed sets with respect to the Wasserstein metric.

\begin{lemma}
\label{the:MartingaleClosed} 
The sets $\mM_2(\R^d)$ and $\mM^*_2(\R^d)$ are closed in
$\mP_1(\R^d\times\R^d)$.
\end{lemma} 
\begin{proof} 
Assume that
$\mu_n \in \mM^*_2(\R^d)$ and $\mu \in \mP_1(\R^d\times\R^d)$ such that
$d_W(\mu_n,\mu) \to 0$.
Then $\mu_n \to \mu$ in distribution and $(\mu_n)$ is uniformly integrable. 
Let $\phi: \R^d \to \R_+$
be continuous and bounded. By 
Lemma \ref{lem:submart-char},
it is sufficient to verify that
\begin{equation}
 \label{eq:MartingaleLimit}
 \int_{\R^d\times\R^d} (x_2-x_1) \phi(x_1)  \mu(\ud x) \ge 0.
\end{equation}

To do this, let $g(x) \defeq (x_2-x_1) \phi(x_1)$,
fix $t > 0$ and choose a continuous function 
$k_t: \R^d\times\R^d \to [0,1]$ 
such that $k_t(x) = 1$
for $|x| \le t$ and $k_t(x) = 0$ for $|x| > t+1$. 
Let us write $g = g^0_t + g^1_t$ where $g^0_t(x) \defeq g(x) k_t(x)$ 
and $g^1_t(x) \defeq g(x) \big(1-k_t(x)\big)$.
Then
\[
 \mu_n(g) - \mu(g)
 = \big(\mu_n(g^0_t) - \mu(g^0_t)\big) + \big(\mu_n(g^1_t) -
 \mu(g^1_t)\big).
\]
Now $g^0_t$ is continuous and bounded, so that $\mu_n(g^0_t) \to \mu(g^0_t)$ by 
convergence in distribution.
Moreover, $|g^1_t(x)| \le 2|x| \, ||\phi||_\infty \charfun{|x| > t}$. This bound together with
uniform integrability shows that $\sup_{n} \big(\mu_n(g^1_t) - \mu(g^1_t)\big)\to 0$ 
as $t \to \infty$.
We can make the last two terms on the right side above arbitrarily close to zero by choosing $t$ large enough,
uniformly in $n$. Then by letting $n \to \infty$
we may conclude that $\mu_n(g) \to \mu(g)$ as $n \to \infty$.
The submartingale property implies 
by Lemma \ref{lem:submart-char}
that $\mu_n(g) \ge 0$ for all $n$, so we conclude that $\mu(g)\ge 0$ and
therefore~\eqref{eq:MartingaleLimit} is valid.

The proof that $\mM_2(\R^d)$ is closed is identical, with equality in
\eqref{eq:MartingaleLimit}.
\end{proof} 


\subsection{Measurability of the coupling map} 
\label{sec:set-valued} 

In what follows, we consider set-valued mappings
(a.k.a.~multifunctions \cite{srivastava}) from a measurable space
$(S,\mS)$ to the topological space $\mP_1\big((\R^d)^n\big)$ equipped
with the Wasserstein metric. A set-valued mapping $G$ maps a point
$\theta \in S$ to a set $G(\theta) \subset \mP_1\big((\R^d)^n\big)$. The
\emph{set-valued inverse} of such a mapping $G$ is defined by
\[
 G^-(A)
 \defeq \{\theta \in S\given G(\theta) \cap A \neq \emptyset\},
 \qquad A \subset \mP_1\big((\R^d)^n\big).
\]
The set-valued map $G$ is called \emph{measurable} if $G^-(A) \in \mS$
for all closed $A \subset \mP_1\big((\R^d)^n\big)$. By expressing an
open set $U \subset \mP_1\big((\R^d)^n\big)$ as a countable union of 
closed balls,
we see that the measurability of $G$ implies that $G^-(U) \in \mS$
also for open sets $U$.

\begin{proposition} 
 \label{prop:couplling-measurable} 
Let $P$ and $Q$ be probability kernels from $S$ to $\R^d$ with finite
first moments. Then
\[
 F(\theta) \defeq \Gamma \big( P_\theta, Q_\theta \big)
\]
is measurable as a set-valued mapping from $S$ to $\mP_1(\R^d\times\R^d)$.
\end{proposition} 

The proof of Proposition~\ref{prop:couplling-measurable} is based on
the three auxiliary lemmas which will be stated and proved next.

\begin{lemma} 
\label{the:KernelStrongMeasurable} 
Let $P$ be probability kernel from $S$ to $\R^d$ with finite first moments.
Then $\theta \mapsto P_\theta$ is a measurable map from $S$ to
$\mP_1(\R^d)$.
\end{lemma} 
\begin{proof} 
Let us first verify that $\theta \mapsto P_\theta f$ is measurable for
every Borel function $f: \R^d \to \R$ such that $\int |f(y)|
P_\theta(\ud y) < \infty$ for all $\theta \in S$. Choose a sequence of
simple Borel functions such that $f_n \to f$ and $|f_n| \le |f|$
pointwise.  By linearity, $\theta \mapsto P_\theta f_n$ is measurable
for any $n$. By dominated convergence,
\[
 P_\theta f = \lim_{n \to \infty} P_\theta f_n
\]
by which $\theta \mapsto P_\theta f$ is measurable as a pointwise
limit of measurable functions.

Let then $B_\epsilon(\mu)$ denote the closed $d_W$-ball
with radius $\epsilon > 0$ and centre $\mu \in \mP_1(\R^d)$.
We will next show that the preimages $A_{\epsilon,\mu} \defeq
\{\theta: P_\theta \in B_\epsilon(\mu) \}$ of closed balls are
measurable.
By Lemma \ref{lem:wasserstein-countable} in Appendix
\ref{sec:wasserstein}, there exists a countable set $\mathcal{T}_d$ of
1-Lipschitz functions on $\R^d$ such that
\[
 A_{\epsilon,\mu}
 = \Big\{ \theta: \sup_{g\in \mathcal{T}_d} 
   \big[ P_\theta(g)-\mu(g) \big] \le \epsilon \Big\}
  = \bigcap_{g\in\mathcal{T}_d} 
    \big\{\theta: \big[P_\theta(g)-\mu(g)\big] \le \epsilon \big\}.
\]
Therefore, $A_{\epsilon,\mu}$ is measurable as a countable
intersection of measurable sets.

Let then $U$ be an open set in $\mP_1(\R^d)$. Because $\mP_1(\R^d)$ is a 
separable metric space,
$U$ may be expressed as a countable union of $d_W$-balls 
$B_1,B_2,\ldots$, and therefore
\[
 \{\theta: P_\theta \in U\} = \bigcup_{i=1}^\infty
 \{\theta: P_\theta \in B_i\}
\]
is measurable. This implies the claim.
\end{proof}

We next consider the marginaliser map 
$\Pi: \mP\big((\R^d)^n\big) \to \mP(\R^d)^n$ defined by
\[
 \Pi(\mu) = (\pi^1_\# \mu, \ldots, \pi^n_\# \mu).
\]
It takes a probability measure on $(\R^d)^n$ as its input and returns 
its marginal distributions on $\R^d$. If the input of $\Pi$ has a
finite first moment, then so do its its marginal distributions.
Therefore, we may also consider $\Pi$ as a mapping from
$\mP_1\big((\R^d)^n\big)$
onto $\mP_1(\R^d)^n$.

\begin{lemma} 
\label{the:BorelImages} 
Let $S$ and $S'$ be Polish spaces and $f: S \to S'$ a Borel map such that 
$f^{-1}(y)$ is compact for all $y \in S'$. Then $f$ maps closed sets into Borel sets.
\end{lemma} 
\begin{proof} 
By \cite[Propositions 3.1.21 and 3.1.23]{srivastava} the graph
of $f$
\[
 \graph(f) \defeq \{(x,f(x)): x \in S\}.
\]
is a Borel set in $S \times S'$. For any 
closed set $A \subset S$, the image $f(A)$ can be represented
as a projection of the set
\[
 B \defeq \graph(f) \cap (A \times S').
\]
Observe next that for any $y \in S'$ the section
\[
 \{ x: (x,y) \in B \} = f^{-1}(y) \cap A
\]
is compact. Therefore, Novikov's theorem 
\cite[Theorem 4.7.11]{srivastava} implies that $f(A)$ is Borel.
\end{proof} 

\begin{lemma}
\label{the:Marginaliser} 
The marginaliser map $\Pi: \mP_1\big((\R^d)^n\big) \to \mP_1(\R^d)^n$ 
defined by
\[
 \Pi(\mu) = (\pi^1_\# \mu, \dots, \pi^n_\# \mu)
\]
maps closed sets into Borel sets.
\end{lemma} 
\begin{proof} 
$\Pi$ is continuous by Lemma~\ref{the:ContinuousProjection}, and hence
also Borel. The spaces $\mP_1\big((\R^d)^n\big)$ and $\mP_1(\R^d)^n$ 
are Polish. The preimage of $\Pi$ for any singleton is compact by
Lemma~\ref{the:CouplingCompact}, because
$\Pi^{-1}(\{\nu_1,\ldots,\nu_n\}) = \Gamma(\nu_1,\ldots,\nu_n)$. The
rest follows from Lemma~\ref{the:BorelImages}.
\end{proof} 

\begin{proof}[Proof of Proposition~\ref{prop:couplling-measurable}] 
We write the set of couplings of $P_\theta$ and 
$Q_\theta$ again as a preimage of the marginaliser,
\[
 F(\theta)
 = \Pi^{-1}\big( \{(P_\theta, Q_\theta)\} \big).
\]
Note that $F(\theta) \cap A \neq \emptyset$ if and only if $\mu \in F(\theta)$
for some $\mu \in A$, that is, $\Pi(\mu) = \big(P_\theta, Q_\theta\big)$ for some $\mu \in A$. Therefore, the set-valued inverse
of $F$ may be written as
\[
 F^-(A)
 = \{\theta \in S\given F(\theta) \cap A \neq \emptyset\}
 = \{\theta\in S\given (P_\theta, Q_\theta) \in
   \Pi(A) \}.
\]
By Lemma~\ref{the:Marginaliser}, $\Pi(A)$ is a Borel set in 
$\mP_1(\R^d) \times \mP_1(\R^d)$ whenever $A \subset
\mP_1(\R^d\times\R^d)$ is closed.
By Lemma~\ref{the:KernelStrongMeasurable}, the maps $\theta \mapsto 
P_\theta$ and $\theta \mapsto Q_\theta$ are measurable from $S$ to 
$\mP_1(\R^d)$. 
Thus also the map $\theta \mapsto (P_\theta,Q_\theta)$ 
is measurable from $S$ to $\mP_1(\R^d)\times\mP_1(\R^d)$.
We may hence conclude that $F^-(A)$ is a measurable subset of 
$S$ for any closed $A \subset \mP_1(\R^d\times\R^d)$.
\end{proof} 


\subsection{Proof of Theorem \ref{th:strassenKernel}} 
\label{sec:proof-conclusion} 

Assume that $P_\theta \lecx Q_\theta$ for all $\theta \in S$. Consider
the set-valued mapping $G(\theta) \defeq F(\theta) \cap \mM$, where
$F(\theta) = \Gamma(P_\theta, Q_\theta)$ is the set of couplings of
$P_\theta$ and $Q_\theta$, and $\mM \defeq \mM_2(\R^d)$ is the
collection of joint distributions of two-parameter martingales.
Proposition \ref{prop:couplling-measurable} shows that $F$ is a
measurable set-valued mapping from $S$ to the subsets of
$\mP_1(\R^d\times\R^d)$. For any $A \subset \mP_1(\R^d\times\R^d)$, 
the set-valued inverse of $G$ can be written as
\[
 G^{-}(A) = F^-\big(\mM \cap A).
\]
Because $\mM$ is closed by Lemma \ref{the:MartingaleClosed}, we see
that $G$ is a measurable set-valued mapping from $S$ to the subsets of
$\mP_1(\R^d\times\R^d)$. Furthermore, because $F(\theta)$ is compact for all
$\theta$ by Lemma \ref{the:CouplingCompact}, also $G(\theta)$ is
compact for all $\theta$. Hence $G$ is a measurable compact-valued
mapping from $S$ to the subsets of $\mP_1(\R^d\times\R^d)$. Strassen's coupling
characterisation (Theorem \ref{th:strassen}) implies that $G(\theta)$
is nonempty for all $\theta$. A measurable selection theorem of
Kuratowski and Ryll-Nardzewski \cite{kuratowski-ryll-nardzewski} (see
alternatively \cite[Theorem 5.2.1]{srivastava}) now implies that there
exists a measurable selection for $G$, that is, a measurable function
$g: S \to \mP_1(\R^d\times\R^d)$ such that $g(\theta) \in G(\theta)$ for all
$\theta$. Let us now define a map $(\theta,B) \mapsto R_\theta(B)$ by
setting
\[
 R_\theta(B) \defeq \ev_B(g(\theta))
\]
for $\theta \in S$ and Borel sets $B \subset \R^d\times\R^d$, where 
$\ev_B(\mu) = \mu(B)$. Then $R_\theta \in \mM(\R^d\times\R^d)$ is a
coupling of $P_\theta$ and $Q_\theta$ for every $\theta \in S$. We are
left with showing that $\theta \mapsto R_\theta(B)$ is measurable for
any Borel set $B \subset \R^d\times\R^d$.  This follows because the
map $\ev_B: \mP_1(\R^d\times\R^d) \to \R$ is measurable by Lemma
\ref{lem:measures-to-kernel} in Appendix \ref{sec:measures-to-kernel}.
Hence $R$ is a pointwise coupling of the probability kernels $P$ and
$Q$.

If $P_\theta \leicx Q_\theta$ for all $\theta \in S$, then by
repeating the above construction with $\mM$ replaced by $\mM^* :=
\mM^*_2(\R^d)$ we obtain a probability kernel $R$ which is a pointwise
coupling of $P$ and $Q$ such that $R_\theta \in \mM^*$ for all
$\theta \in S$.

Finally we note that if $R$ is pointwise coupling of $P$ and $Q$ such
that $R_\theta \in \mM_2(\R^d)$ (resp.\ $\mM^*_2(\R^d)$) for all $\theta$,
then Theorem~\ref{th:strassen} immediately implies that $P_\theta
\lecx Q_\theta$ (resp.\ $P_\theta \leicx Q_\theta$) for all $\theta$.
\qed


\subsection{Proof of Theorem~\ref{cor:cond-strassen-geni}}
\label{sec:proof-other} 

Let us first prove the forward implication in
\eqref{item:cond-icx-strassen-geni}. Suppose that $X \mid Z \leicx Y
\mid Z$. Let $P$ and $Q$ be regular conditional distributions of $X$
and $Y$ given $Z$, respectively, and denote the distribution of $Z$ by
$\mu$. Then by Proposition~\ref{cor:cond-cx-char}, $P_\theta \leicx
Q_\theta$ for all $\theta \in S$ outside a set of $\mu$-measure zero.
By redefining $P_\theta$ and $Q_\theta$ as equal on this set of
$\mu$-measure zero, we may assume that $P_\theta \leicx Q_\theta$ for
all $\theta \in S$. By Theorem \ref{th:strassenKernel} there exists a
probability kernel $R$ from $S$ to $\R^d\times\R^d$ which is a
pointwise coupling of $P$ and $Q$ and satisfies $R_\theta \in
\mM^*_2(\R^d)$ for all $\theta$.

Let $(\hat{X},\hat{Y},\hat{Z})$ be a random element in $\R^d\times\R^d\times S$
with distribution
\[
 \lambda(\ud x\times \ud y\times \ud \theta) \defeq \mu(\ud \theta) R_\theta(\ud x \times \ud y),
\]
Because $\lambda(\ud x\times\R^d\times\ud \theta) = \mu(\ud \theta)
P_\theta(\ud x)$ and $\lambda(\R^d\times\ud y \times\ud \theta) =
\mu(\ud \theta) Q_\theta(\ud y)$, it follows that $(\hat X, \hat Z)
\deq (X,Z)$ and $(\hat Y, \hat Z) \deq (Y,Z)$. Hence
$(\hat{X},\hat{Y},\hat{Z})$ is a $Z$-conditional coupling of $X$ and
$Y$. We still need to verify that
\begin{equation}
 \label{eq:SubmartAs}
 \hat X \le \E( \hat Y \mid \hat X, \hat Z ) \quad\text{almost surely}.
\end{equation}
For any measurable $A \subset \R^d \times S$, by denoting $A_\theta
\defeq \{x \in \R^d \given (x,\theta)\in A\}$, we see with the help of
Lemma~\ref{lem:submart-char} that
\begin{align}
    \E\big[\E[\hat{Y}\mid
    \hat{X},\hat{Z}] \charfun{(\bigvphantom\smash{\hat{X}},\smash{\hat{Z}})\in
      A}\big]
    &= \E\big[\hat{Y}\charfun{(\bigvphantom\smash{\hat{X}},\smash{\hat{Z}})\in A}\big]
    \nonumber\\
    &= \int \mu(\ud \theta) \int y \charfun{x\in A_\theta} R_\theta(\ud x\times\ud y)
    \nonumber\\
    &\ge \int \mu(\ud \theta) \int x \charfun{x\in A_\theta} R_\theta(\ud x\times\ud y) 
    \label{eq:mart-prop} \\
    &= \E\big[\hat{X}
    \charfun{(\bigvphantom\smash{\hat{X}},\smash{\hat{Z}})\in A}\big],\nonumber
\end{align}
because $R_\theta \in \mM^*_2(\R^d)$ for all $\theta$.
This implies ~\eqref{eq:SubmartAs}.

To prove the other direction in \eqref{item:cond-icx-strassen-geni}, assume next that $(\hat X, \hat
Y, \hat Z)$ is a $Z$-conditional coupling of $X$ and $Y$
satisfying~\eqref{eq:SubmartAs}. Recall
Lemma~\ref{lem:cond-countable-char} and let $f\in\testiconvex$, then
conditional Jensen's inequality implies that
\[
 f(\hat X)
 \le f\big(  \E ( \hat Y \mid \hat X, \hat Z ) \big)
 \le \E \big( f(  \hat Y ) \bigmid \hat X, \hat Z \big)
\]
almost surely. By taking $\hat Z$-conditional expectations on both sides above, it follows that
\[
  \E \big( f( \hat X) \bigmid \hat Z \big)
  \le \E \big( f( \hat Y) \bigmid \hat Z \big).
\]
Because $(\hat X, \hat Z) \deq (X,Z)$ and $(\hat Y, \hat Z) \deq (Y,Z)$, we may remove the hats above to conclude that 
\begin{equation}
  \E \big( f(  X) \bigmid  Z \big)
  \le \E \big( f(  Y) \bigmid  Z \big)
    \label{eq:cond-exp-follows}
\end{equation}
almost surely. By Lemma \ref{lem:cond-countable-char}, this
implies $X \mid Z \leicx Y \mid Z$.

The proof of the forward implication of claim \eqref{item:cond-cx-strassen-geni} is obtained by
imitating the proof of \eqref{item:cond-icx-strassen-geni}; by replacing the inequality in
\eqref{eq:SubmartAs} and \eqref{eq:mart-prop} by equality, and
applying Lemma~\ref{lem:mart-char} in place of
Lemma~\ref{lem:submart-char}. Similarly, the reverse implication of
claim \eqref{item:cond-cx-strassen-geni} is obtained by using $f\in\testconvex$ in place of
$f\in\testiconvex$ in \eqref{eq:cond-exp-follows}.
\qed



\section{Application to pseudo-marginal Markov chain Monte Carlo} 
\label{sec:application} 

We discuss here briefly the application 
which initially motivated the present work.  
The application focuses on Markov chain Monte
Carlo (MCMC) algorithms targeting a probability distribution $\pi$ on
a general state space $\mathsf{X}$.  In particular, the interest lies in the 
so-called pseudo-marginal MCMC with transition probability
\[
    K(x,w;\ud y\times \ud u)
    \defeq q(x,\ud y) Q_y(\ud u) \min\bigg\{1,
      r(x,y) \frac{u}{w}\bigg\} + \ind_{\ud y\times \ud u}(x,t) \rho(x,w),
\]
parametrised by a proposal kernel $(x,B)\mapsto q(x,B)$ on $\mathsf{X}$
and an auxiliary kernel $(x,B)\mapsto Q_x(B)$ from $\mathsf{X}$ to
$\R_+$, satisfying $\int Q_x(\ud w) w = 1$ for every $x\in\mathsf{X}$.
The function $r(x,y)$ is the Radon-Nikodym derivative 
$\frac{\pi(\ud y)}{\pi(\ud x)}\frac{q(y,\ud x)}{q(x,\ud y)}$ 
whenever well-defined, and zero otherwise \cite[cf.][]{tierney-note},
and the `probability of rejection'
$\rho(x,w)\in[0,1]$ is such that $K$ defines a
transition probability. We advise an interested
reader to consult \cite{andrieu-vihola-order} for details and
\cite{andrieu-vihola-pseudo,andrieu-roberts} and references therein
for more thorough introduction to the method.

It is not difficult to check that $K$ is reversible with respect to 
the distribution
\[
    \tilde{\pi}(\ud x\times \ud w)
    = \pi(\ud x) Q_x(\ud w) w,
\]
and it is evident that
$\tilde{\pi}$ admits $\pi$ as its first marginal. 
This means that, if the Markov
chain $(X_k,W_k)_{k\ge 1}$ with transition probability $K$
is irreducible, the ergodic averages
approximate the integral of any $\pi$-integrable function
$f:\mathsf{X}\to\R$:
\begin{equation*}
    \frac{1}{n} \sum_{k=1}^n f(X_k) \xrightarrow{n\to\infty}
    \pi(f) \defeq 
 \int_{\mathsf{X}} f(x) \pi(\ud x).
\end{equation*}

The so-called \emph{asymptotic variance} is a common MCMC efficiency criterion,
which is informative about the asymptotic rate of convergence above.
It is defined in the present setting for any $f\in L^2(\pi) \defeq
\big\{f:\mathsf{X}\to\R \given \pi(f^2)<\infty\}$ through
\begin{equation}
    \sigma^2(K,f) \defeq \lim_{n\to\infty} \E\bigg[ \frac{1}{\sqrt{n}}
    \sum_{k=1}^n \big[f(X'_k) - \pi(f)\big]\bigg]^2,
    \label{eq:asvar}
\end{equation}
where $(X'_k,W'_k)_{k\ge 1}$ is a stationary version of the MCMC
chain---that is, $(X'_1,W'_1)\sim \tilde{\pi}$ and $(X'_k,W'_k)_{k\ge
1}$ follows the transition probability $K$. The limit in
\eqref{eq:asvar} always exists, but can be infinite
\cite{tierney-note}. In the pseudo-marginal context, we are interested
in how the choice of the laws $\{Q_x\}_{x\in\mathsf{X}}$ affects the
asymptotic variance.

The usual method to compare asymptotic variances of reversible
Markov chains is Peskun's theorem 
\cite{peskun} and its generalisations
\cite{tierney-note,caracciolo-pelissetto-sokal,leisen-mira}.
It states that if two Markov transition probabilities 
$K$ and $K'$ are reversible with
respect to the same probability distribution $\mu$, then
\[
    \sigma^2(K,f) \le \sigma^2(K',f) 
    \qquad\text{for all $f\in L^2(\mu)$,}
\]
if and only if
\[
    \langle g, K g \rangle_\mu \le 
    \langle g, K' g \rangle_\mu \qquad\text{for all $g\in L^2(\mu)$,}
\]
where $\langle f, g \rangle_\mu \defeq \int f(x) g(x)\mu(\ud x)$.
This is inapplicable in the present application, as the two Markov
chains $K$ and $K'$ with $\{Q_x\}_{x\in\mathsf{X}}$ and
$\{Q'_x\}_{x\in\mathsf{X}}$ are reversible with respect to different
invariant distributions $\tilde{\pi}$ and
$\tilde{\pi}'$, respectively.

Because the interest lies only in functions which are constant in the
second coordinate, it is still possible to pursue such an ordering.
Indeed, given a pointwise martingale coupling $R_x$ of $Q_x$ and
$Q'_x$, which exists by Theorem \ref{th:strassenKernel} if $Q_x \lecx
Q'_x$ for all $x\in\mathsf{X}$, it turns out to be possible to deduce
a `Peskun-type' order of the asymptotic variances \cite[Theorem
10]{andrieu-vihola-order}
\begin{equation}
    \sigma^2(K,f)\le \sigma^2(K',f)\quad\text{for all}\quad
    f(x,w)=f(x) \in L^2(\pi).
    \label{eq:pm-var-order}
\end{equation}
We will next briefly summarise why a strong
martingale coupling as in Theorem \ref{th:strassenKernel} is
fundamental to prove this result. 

The key of the proof of \eqref{eq:pm-var-order}
relies in `embedding' the two Markov kernels $K$ and
$K'$ on a common Hilbert space.
The martingale coupling allows to construct the following Markov kernels 
$\breve{K}$ and $\breve{K}'$ and a distribution 
$\breve{\pi}$:
\begin{align*}
    \breve{K}(x,w,v;\ud y\times\ud u\times\ud t)
    &=q(x,\ud y)R_y(\ud u\times\ud t)\frac{t}{u}
    \min\Big\{1,r(x,y)\frac{u}{w}\Big\} \\
    &\phantom{=}+ \ind_{\ud y\times\ud
      u\times\ud t}(x,w,v) \rho(x,w) \\
    \breve{K}'(x,w,v;\ud y\times\ud u\times\ud t)
    &=q(x,\ud y)R_y(\ud u\times\ud t)
    \min\Big\{1,r(x,y)\frac{t}{v}\Big\} \\
    &\phantom{=}+ \ind_{\ud y\times\ud
      u\times\ud t}(x,w,v) \rho'(x,v) \\    
    \breve{\pi}(\ud x\times \ud t\times \ud u)
    &= \pi(\ud x) R_x(\ud w \times \ud v) v.
\end{align*}
It is not difficult to check that both $\breve{K}$ and $\breve{K}'$
are reversible with respect to $\breve{\pi}$, and $\breve{\pi}$
coincides marginally with $\tilde{\pi}$ and $\tilde{\pi}'$ so that
$\tilde{\pi}(\ud x\times \ud w) = \breve{\pi}(\ud x\times \ud w \times
\R_+)$ and $\tilde{\pi}'(\ud x\times \ud v) = \breve{\pi}(\ud x\times
\R_+ \times \ud v)$. Similarly, the kernels $\breve{K}$ and
$\breve{K}'$ coincide marginally with $K$ and $K'$; see \cite[Lemma
20]{andrieu-vihola-order}. This construction enables the Hilbert
space techniques, on $L^2(\breve{\pi})$, to be used. The martingale
coupling allows to show that \cite[Theorem 22(b)]{andrieu-vihola-order}
\[
    \langle g, \breve{K} g \rangle_{\breve{\pi}} \le 
    \langle g, \breve{K'} g \rangle_{\breve{\pi}} 
    \qquad\text{for all $g(x,w,v) =
      g(x,w)$ with $g\in
      L^2(\tilde{\pi})$,}
\]
which ultimately leads to the order $\sigma^2(K,f) \le \sigma^2(K',f)$
for all $f(x,w)=f(x)$ with $f\in L^2(\pi)$.


\section{Extensions and implications} 
\label{sec:extensions} 

We discuss next some extensions and implications of our results.
In Strassen's original paper, Theorem \ref{th:strassen} is formulated 
for countably many distributions instead of a pair. 
Extension of Theorems~\ref{th:strassenKernel} and \ref{cor:cond-strassen-geni}
into a context with countably many kernels
is straightforward. For instance, we may
formulate the following result.
\begin{proposition} 
    \label{prop:countable-gen} 
Suppose that for each $i\in\N$, 
$(\theta,B)\mapsto P_\theta^{(i)}(B)$ is a probability kernel from $S$ to $\R^d$.
\begin{enumerate}[(i)]
\item \label{item:countable-coupling-cx} 
  $P_\theta^{(i)}\lecx P_\theta^{(i+1)}$ for all $i\in\N$ if and only
if there exists a pointwise coupling $R$ of
$\{P_\theta^{(i)}\}_{i\in\N}$ such that
$R_\theta\in\mM_{\N}(\R^d)$. 
\item \label{item:countable-coupling-icx}
  $P_\theta^{(i)}\leicx P_\theta^{(i+1)}$ for all $i\in\N$ if and only
if there exists a pointwise coupling $R$ of
$\{P_\theta^{(i)}\}_{i\in\N}$ such that
$R_\theta\in\mM^*_{\N}(\R^d)$. 
\end{enumerate}
More precisely, $R$ above is a kernel from $S$ to
$(\R^d)^{\N}$ such that $R_\theta(\uarg)$ is the law of the
$\R^d$-valued (sub-)martingale $(X_\theta^{(i)})_{i\ge 1}$ such that
$\P(X_\theta^{(i)}\in A) = P_\theta^{(i)}(A)$. 
\end{proposition} 
\begin{proof} 
For, \eqref{item:countable-coupling-cx} assume that for each $i\in\N$
$P_\theta^{(i)}\lecx P_\theta^{(i+1)}$ and let
$R_\theta^{(i)}$ stand for their pointwise coupling.
There exist kernels $T^{(i)}$ from $S\times \R^d$ to $\R^d$ (regular conditional
probabilities) such that
\[
    R_\theta^{(i)}(\ud x_{i-1}\times \ud x_i) = P_\theta^{(i-1)}(\ud
    x_{i-1}) T^{(i)}_{\theta, x_{i-1}}(\ud x_i).
\]
We may define $R_\theta$ inductively through its finite-dimensional
distributions by letting $R_\theta(\ud x_1\times \ud x_2\times\R^{\N}) =
    R_\theta^{(2)}(\ud x_1\times \ud x_2)$ and for $i\ge 3$
\[
     R_\theta(\ud x_1\times\cdots \times \ud x_i\times\R^{\N})
    = R_\theta(\ud x_1\times\cdots \times \ud x_{i-1}\times\R^{\N})
    T_{\theta, x_{i-1}}^{(i)}(\ud x_i).
\]
The other direction follows from Jensen's inequality.
The proof of \eqref{item:countable-coupling-icx}
follows similar lines.
\end{proof} 

The following characterisation of increasing convex orders in terms of
convex stochastic order and strong stochastic order \cite[Theorem
  3.4.3]{muller-stoyan} is sometimes convenient.
\begin{theorem} 
    \label{thm:icx-strong-cx} 
If $X\leicx Y$ then there exist a probability space with random variables
$\hat{X},\hat{W},\hat{Y}$ such that $\hat{X}\eqd X$, $\hat{Y}\eqd Y$, 
$\hat{X}\le \hat{W}$ almost surely and $\hat{W}\lecx \hat{Y}$.
\end{theorem} 
We record the following result, which is a conditional version 
Theorem \ref{thm:icx-strong-cx}.
\begin{proposition} 
If $X\mid Z \leicx Y\mid Z$ then there exist a probability space with 
random variables
$\hat{X}$, $\hat{Y}$, $\hat{Z}$ and $\hat{W}$ such that 
$(\hat{X},\hat{Z})\eqd (X,Z)$, 
$(\hat{Y},\hat{Z})\eqd (Y,Z)$, 
$\hat{X} \le \hat{W}$ almost surely and 
$\hat{W}\mid \hat{Z}\lecx \hat{Y}\mid\hat{Z}$.
\end{proposition} 
\begin{proof} 
We may take the triple $(\hat{X},\hat{Y},\hat{Z})$ from 
Theorem \ref{cor:cond-strassen-geni} and 
set $\hat{W}\defeq \E(\hat{Y}\mid\hat{X},\hat{Z})$.
Then $\hat{X}\le \hat{W}$, and
for any convex $\phi:\R^d\to\R$, Jensen's inequality yields 
$    \E\big(\phi(\hat{W})\bigmid \hat{Z}\big)
    \le 
    \E\big(\phi(\hat{Y}) \bigmid \hat{Z}\big).$
\end{proof} 

We next turn into so-called martingale optimal transport problem
\cite{beiglboeck-juillet}. which is linked to applications in
mathematical
finance~\cite[e.g.][]{dolinsky-soner,beiglbock-henry-labordere-penkner,galichon-henry-labordere-touzi,hobson-neuberger}.
Optimal transport problems, in general, mean finding a coupling of two
probability measures $\mu$ such that the `cost' $\mu(c) \defeq \iint
c(x,y)\mu(\ud x\times \ud y)$ is minimised. Usually the minimisation
is over all couplings, but in the martingale optimal transport the
minimisation is constrained to martingale couplings.

We illustrate that when a parametric version of such a problem is
considered, our results allow to ensure that minimisers
can be chosen in a measurable manner in this context.
\begin{proposition} 
    \label{prop:measurable-optimal-transport} 
Suppose that $P$ and $Q$ are probability kernels
from $S$ to $\R^d$
and $c:\R^d\times\R^d\to(-\infty,\infty]$ is lower
semi-continuous and satisfies the lower bound
\begin{equation*}
    c(x,y) \ge - C(1+|x|+|y|) \qquad\text{for all $x,y\in\R^d$}
\end{equation*}
for some finite constant $C$.
\begin{enumerate}[(i)]
    \item 
      \label{item:martingale-transport}
      If $P_\theta \lecx Q_\theta$ for all $\theta$, 
then there exists a measurable optimal
martingale transport plan $\gamma$, that is,
a kernel $(\theta,B)\mapsto \gamma_\theta(B)$ from $S$ to
$\R^d\times\R^d$ such that for every $\theta$, $\gamma_\theta$ is a
martingale coupling of $P_\theta$ and $Q_\theta$ which minimises
$\mu(c)$ over all martingale couplings $\mu$ of $P_\theta$ and $Q_\theta$.
    \item 
      \label{item:transport}
      If $P_\theta,Q_\theta\in\mP_1(\R^d)$, then
      there exists a measurable optimal transport plan
      $\gamma^*$,
      that is, a kernel $(\theta,B)\mapsto \gamma^*_\theta(B)$ from $S$ to
$\R^d\times\R^d$ such that for every $\theta$, $\gamma^*_\theta$ is a
coupling of $P_\theta$ and $Q_\theta$ which minimises
$\mu(c)$ over all couplings $\mu$ of $P_\theta$ and $Q_\theta$.
\end{enumerate}
\end{proposition} 
\begin{proof} 
Consider first \eqref{item:martingale-transport}, and
denote for brevity $\Gamma(\theta)\defeq 
\Gamma(P_\theta,Q_\theta)\cap \mM_2(\R^d)$.
Recall that $\theta\to \Gamma(\theta)$ is 
compact-valued and measurable;
see the proof of Theorem~\ref{cor:cond-strassen-geni}.
Denote $v_\theta \defeq \inf_{\mu \in \Gamma(\theta)}
    \mu(c)$, and let us check that $v_\theta>-\infty$. 
Let $\mu_1,\mu_2,\ldots\in \Gamma(\theta)$, then
  because $\Gamma(\theta)$ is compact, there exists a
convergent subsequence $\mu_n'\to\mu$. By assumption,
\[
    c_-(x,y) \defeq -\min\{c(x,y),0\}
    \le C (1 + |x| + |y|),
\]
implying that $c_-$ is uniformly integrable with respect to $(\mu_n')$.
Lemma 5.17 of \cite{ambrosio-gigli-savare} states that 
then $\liminf_{n\to\infty} \mu_n'(c) \ge \mu(c) > -\infty$.

Define next for each $q\in\Q$ the level sets $L_q \defeq 
\{\mu\in\mP_1(\R^d\times\R^d)
      \given \mu(c)\le q\}$, which are closed following the argument
    above.
The set-valued mapping
\[
    \tilde{L}_q(\theta) \defeq \begin{cases}
    L_q\cap \Gamma(\theta),&\text{if }L_q\cap\Gamma(\theta)\neq \emptyset, \\
    \Gamma(\theta),&\text{otherwise},
    \end{cases}
\]
is compact-valued. Let us turn next
into showing that $\theta\mapsto \tilde{L}_q(\theta)$ is a measurable
as a set-valued mapping, by considering the set-valued inverse
of a closed $F$
\begin{align*}
    \tilde{L}_q^-(F)
    &= \{\theta\in S\given \tilde{L}_q(\theta)\cap F \neq \emptyset\} \\
    &= \{\theta\in S: \Gamma(\theta) \cap L_q \cap F \ne \emptyset \} \cup
    \{\theta: \Gamma(\theta) \cap F \ne \emptyset, \, \Gamma(\theta)
      \cap L_q = \emptyset\} \\
    &= \Gamma^-(F\cap L_q) \cup \big( \Gamma^-(F) \setminus
    \Gamma^-(L_q)\big),
\end{align*}
which is measurable due to the measurability of $\theta\mapsto
\Gamma(\theta)$.

It is straightforward to check that
\[
    \Gamma_\mathrm{opt}(\theta) \defeq \{\mu\in \Gamma(\theta)\given \mu(c) =
      v_\theta\}
    = \bigcap_{q\in\Q} L_q(\theta).
\]
Because $\Gamma_\mathrm{opt}(\theta)$ is a countable intersection of
compact-valued mappings $\theta \to \tilde{L}_q(\theta)$, and because
$\mP_1(\R^d\times\R^d)$ is a complete separable metric space, 
it follows \cite[Theorems 3.1 and
4.1]{himmelberg} that it is measurable
as a set-valued mapping. Because $\Gamma_\mathrm{opt}$ is non-empty,
compact-valued and measurable,
we may apply the measurable selection theorem 
\cite{kuratowski-ryll-nardzewski} and the evaluation map $\ev_B$ 
as in the proof of Theorem~\ref{cor:cond-strassen-geni} to conclude
the existence of the desired $\gamma_\theta$.

The proof of \eqref{item:transport} is similar, because also
$\theta\to \Gamma^*(\theta) \defeq \Gamma(P_\theta,Q_\theta)$ 
is compact-valued and measurable by Lemma
\ref{the:CouplingCompact} and Proposition
\ref{prop:couplling-measurable}.
\end{proof} 

We record the following remarks about Proposition
\ref{prop:measurable-optimal-transport}:
\begin{enumerate}[(i)] 
  \item The results may be
  extended into countably many $P_\theta^{(i)}$ in similar lines as
  Proposition \ref{prop:countable-gen}.
    \item The assumptions on the cost function $c$ 
coincide with those of Beiglböck,
Henry-Labord\`ere and Penkner
\cite{beiglbock-henry-labordere-penkner}, who consider the 
martingale optimal transport problem in the scalar case.
\item Proposition \ref{prop:measurable-optimal-transport}
\eqref{item:transport} is probably well-known, but we included it for
completeness. Indeed, Corollary 5.22 of Villani 
\cite{villani} is similar, without the integrability assumption on
$P_\theta$ and $Q_\theta$, but with constant lower bound and continuity assumption 
on $c$.
\end{enumerate} 


\section*{Acknowledgements} 

We thank Eero Saksman, Julio Backhoff and a referee for useful remarks.
L.~Leskel\"a was working at the University of Jyv\"askyl\"a while the major
part of this article was written.
M.~Vihola was supported by the Academy of Finland projects 250575 and
274740.


\appendix 

\section{Dense countable families of convex functions}
\label{sec:countable-char} 

\begin{proof}[Proof of Lemma \ref{lem:countable-char}] 
We can take $\testconvex$ as the countable family of max-affine
convex functions $f:\R^d\to\R$ taking the form
\[
    f(x) = \max\{\alpha_1^T x + \beta_1, \ldots, \alpha_n^T x +
      \beta_n\},
\]
where $n\in\N$ and $\alpha_i,\beta_i\in\Q^d$.

To confirm this, take first a non-negative convex $\phi:\R^d\to\R_+$
with $\E \phi(X)<\infty$.
It is not difficult to see that 
for any $\epsilon>0$, we may find a piecewise linear function $g$ defined
as an infinite maximum of affine functions with $\alpha_i,\beta_i\in\Q^d$
\[
    g(x) = \max\{\alpha_i^T x + \beta_i\given i\in\N\},
\]
such that $|g(x)-\phi(x)|\le\epsilon/2$ for all $x\in\R^d$.
Consequently, $|\E g(x) - \E \phi(x) | \le \epsilon/2$.
Taking 
\[
    g_n(x) \defeq \max\{\alpha_i^T x + \beta_i \given i=1,\ldots,n\},
\]
then $g_n\in\testconvex$ and $g_n(x) \uparrow g(x)$ pointwise.
We conclude by monotone convergence that there exists 
$g_m\in\testconvex$ such that $|\E \phi(x) - \E g_m(x) |\le\epsilon$.

For general $\phi:\R^d\to\R_+$, with $\E \phi(X)$ finite, 
it is sufficient to observe that 
\[
    \lim_{n\to\infty} \E \max\{\phi(x), -n\}
    = \E \phi(x),
\]
and then one can take any $\phi_m = (\phi(x)-m)_+$ and apply the result
above to conclude the existence of $f\in\testconvex$ such that $|\E
\phi(X) - \E f(X) |$ is arbitrarily small.

Similarly one can take $\testiconvex$ as the 
set of increasing $f\in \testconvex$.
\end{proof} 


\section{Wasserstein distance as a countable supremum}
\label{sec:wasserstein} 

Let $\Lip_1(\R^d)\defeq\{f:\R^d\to\R\given |f(x)-f(y)|\le |x-y| \text{ for all }
  x,y\in\R^d\}$ stand for the set of 1-Lipschitz functions on $\R^d$.
\begin{lemma} 
    \label{lem:wasserstein-countable} 
There exists a countable subset $\mathcal{T}_d\subset\Lip_1(\R^d)$
such that
\begin{equation}
    d_W(\mu,\nu) =
    \sup_{f\in\Lip_1(\R)} 
    \big[\mu(f)-\nu(f)\big]
    = \sup_{g\in \mathcal{T}_d} \big[\mu(g)-\nu(g)\big]
    \label{eq:wasserstein-variational}
\end{equation}
for all $\mu, \nu \in \mP_1(\R^d)$.
\end{lemma} 
\begin{proof} 
The first equality in \eqref{eq:wasserstein-variational} 
is the Kantorovich-Rubinstein duality \cite[Remark 6.5]{villani}.
Inspired by \cite[Theorem 3.1.5]{stroock}, we may take
$\mathcal{T}_d$ as all functions of the form
\begin{align}
    g(x) &= \min\{g_1(x), g_2(x),\ldots, g_n(x)\} 
    \label{eq:g-def} \\
    g_k(x) &= q_k + |x-y_k|\nonumber
\end{align}
where $n\in\N$, $q_k\in\Q$ and $y_k\in\Q^d$.

Clearly $\mathcal{T}_d \subset \Lip_1(\R^d)$,
and for any function $f\in\Lip_1(\R^d)$, any compact set $K\subset\R^d$
and $\epsilon>0$ there exists $g\in\mathcal{T}_d$ such that
\[
    \sup_{x\in K} |f(x) - g(x)|\le \epsilon.
\]
Namely, take $y_1,\ldots,y_n$ such that for all $x\in K$ there exists
$m(x)$  such that $|x-y_{m(x)}|\le \epsilon/3$,
and choose $q_k$ such that $0 \le q_k - f(y_k) \le \epsilon/3$.
Then, $g(x)\ge f(x)$ for all $x\in\R^d$, and for any $x\in K$ we have
\begin{align*}
    |g(x) - f(x)| 
    &= \big(g(x) - g(y_{m(x)}) \big)
    + \big(g(y_{m(x)}) - f(y_{m(x)})\big) 
     + \big(f(y_{m(x)}) - f(x)\big) \\
    &\le \frac{\epsilon}{3} + \big(q_{m(x)} - f(y_{m(x)})\big)
    + \frac{\epsilon}{3}.
\end{align*}

Fix then $\mu,\nu\in \mP_1(\R^d)$ and $f\in \Lip_1(\R^d)$, which we
may assume without loss of generality to satisfy $f(0)=0$.
For any $\epsilon>0$ we may find $M<\infty$ such that 
\[
    \int_{|x|>M} |x|  \big(\mu(\ud x) + \nu(\ud x)\big)
    \le \frac{\epsilon}{8},
\]
because $\mu,\nu$ are integrable. Let
$g\in \mathcal{T}_d$ such that $|f(x)-g(x)|\le \epsilon/8$ 
for all $|x|\le M$. Then, 
\[
    \mu(g)-\nu(g)
    \ge \mu(f) - \nu(f) - \mu(|f-g|) - \nu(|f-g|)
    \ge \mu(f) - \nu(f) - \epsilon,
\]
because 
\begin{align*}
    \mu(|f-g|) &\le \frac{\epsilon}{8} + \int_{|x|>M}( \big(|f(x) +
    |g(x)| \big)\mu(\ud x) \\
    &\le \frac{\epsilon}{8} + 2\int_{|x|>M}|x| \mu(\ud x)
    + |g(0)|,
\end{align*}
so $\mu(|f-g|) \le \epsilon/2$ and similarly $\nu(|f-g|)\le \epsilon/2$.
\end{proof} 


\section{From Measure-valued mappings to kernels}
\label{sec:measures-to-kernel} 

\begin{lemma} 
\label{lem:measures-to-kernel} 
For any Borel set $B \subset \R^d$, the evaluation map $\ev_B: \mu \mapsto \mu(B)$ from $\mP_1(\R^d)$
to $\R$ is measurable with respect to the Borel sigma-algebra generated by the
Wasserstein metric on $\mP_1(\R^d)$.
\end{lemma} 
\begin{proof} 
Assume first that $B$ is open. Let $f_n$ be bounded positive 
continuous functions such that $f_n \uparrow \ind_B$ pointwise;
such functions exist by Urysohn's lemma.
        
Note that for each $n$, the map $\Phi_n: \mP_1(\R^d) \to \R$ defined by $\Phi_n(\mu) = \mu(f_n)$ is continuous and thus measurable.
Furthermore, the monotone convergence theorem implies that
$\Phi_n(\mu) \uparrow \ev_B(\mu)$ for every $\mu$ in $\mP_1(\R^d)$. Thus the map $\ev_B$ is
measurable, being a pointwise limit of measurable maps.

We next show that the claim holds for any Borel set. Denote by $\mE$ the
collection of Borel sets $B \subset \R^d$ such that $\ev_B$ is measurable.
If $A,B\in \mE$ and $A\subset B$, then $\ev_{B\setminus A}(\mu) = \ev_{B}(\mu) - \ev_{A}(\mu)$, so
$B \setminus A \in \mE$. Similarly, one can show that 
$\mE$ is closed under monotone unions, and clearly $\R^d \in \mE$.
We conclude that $\mE$ is a Dynkin's $\lambda$-system which contains the open
sets of $\R^d$. Because the collection of open sets is closed under finite intersections,
an application of a monotone class theorem \cite[Theorem 1.1]{kallenberg} shows that
$\mE$ contains all Borel sets of $\R^d$.
\end{proof}


\bibliographystyle{abbrv}
\bibliography{refs.bib}

\end{document}